\documentclass[11pt]{amsart}

\usepackage[T1]{fontenc}
\usepackage[latin1]{inputenc}
\usepackage[english]{babel}
\usepackage{amsmath,amssymb,amsthm}
\usepackage{enumitem}
\usepackage{mymacros}
\usepackage{hyperref}
\usepackage{xcolor}

\newcommand{\capac}[2]{\operatorname{cap}(#1,#2)}
\newcommand{\Hardys}{H^2(\mathbb B_d)}
\newcommand{\inter}{\operatorname{int}}

\renewcommand{\MR}[1]{}

\title{Totally null sets and capacity in Dirichlet type spaces}

\author{Nikolaos Chalmoukis}
\address{Dipartimento di Matematica, Universit\`a di Bologna, 40126, Bologna, Italy}
\email{nikolaos.chalmoukis2@unibo.it}
\thanks{N.C. was supported by the fellowship INDAM-DP-COFUND-2015 ``INdAM Doctoral Programme in Mathematics and/or Applications cofund by Marie Sklodowska-Curie Actions'' \# 713485.}

\author{Michael Hartz}
\address{Fachrichtung Mathematik, Universit\"at des Saarlandes, 66123 Saarbr\"ucken, Germany}
\email{hartz@math.uni-sb.de}
\thanks{M.H. was partially supported by a GIF grant.}


\begin{document}

\begin{abstract}
 In the context of Dirichlet type spaces on the unit ball of $\mathbb{C}^d$, also known as
 Hardy--Sobolev or Besov--Sobolev spaces,
 we compare two notions of smallness for compact subsets of the unit sphere.
  We show that the functional analytic notion of being totally null agrees with the potential theoretic notion of having capacity zero.
  In particular, this applies to the classical Dirichlet space on the unit disc and logarithmic capacity.
  In combination with a peak interpolation result of Davidson and the second named author,
  we obtain strengthenings of boundary interpolation theorems of Peller and 
  Khrushch\"{e}v and of Cohn and Verbitsky.
\end{abstract}

\maketitle

\section{Introduction}

\subsection{Background}
Let $\mathbb{D} \subset \mathbb{C}$ denote the open unit disc and let
\begin{equation*}
  H^2 = \Big\{ f \in \mathcal{O}(\mathbb{D}) : \sup_{0 \le r < 1} \int_{0}^{2 \pi} |f (r e^{ i t})|^2 \, dt < \infty \Big \}
\end{equation*}
be the classical Hardy space.
It is well known that $H^2$ can be identified with the closed subspace of all functions in
$L^2(\partial \mathbb{D})$
whose negative Fourier coefficients vanish.
Correspondingly, subsets of $\partial \mathbb{D}$ of linear Lebesgue
measure zero frequently play the role of small or negligible sets
in the theory of $H^2$ and related spaces.
For instance, a classical theorem of Fatou shows that every function in $H^2$
has radial limits outside of a subset of $\partial \mathbb{D}$ of Lebesgue measure zero;
see for instance \cite[Chapter 3]{Hoffman62}.
For the disc algebra
\begin{equation*}
  A(\mathbb{D}) = \{ f \in C (\overline{\mathbb{D}}): f \big|_{\mathbb{D}} \in \mathcal{O}(\mathbb{D}) \},
\end{equation*}
the Rudin--Carleson theorem shows that
every compact set $E \subset \partial \mathbb{D}$ of Lebesgue measure zero is an interpolation set for $A(\mathbb{D})$,
meaning that for each $g \in C(E)$, there exists $f \in A(\mathbb{D})$ with $f \big|_E = g$.
In fact, one can achieve that $|f(z)| < \|g\|_\infty$ for $z \in \overline{\mathbb{D}} \setminus E$
(provided that $g$ is not identically zero); this is called peak interpolation.
In particular, there exists $f \in A(\mathbb{D})$ with $f \big|_E = 1$ and $|f(z)| < 1$
for $z \in \overline{\mathbb{D}} \setminus E$, meaning that $E$ is peak set
for $A(\mathbb{D})$. Conversely, every peak set and every interpolation set has Lebesgue measure
zero. For background on this material, see \cite[Chapter II]{Gamelin69}.

In the theory of the classical Dirichlet space
\begin{equation*}
  \mathcal{D} = \Big\{ f \in \mathcal{O}(\mathbb{D}): \int_{\mathbb{D}} |f'|^2 \, dA < \infty \Big\},
\end{equation*}
where $A$ denotes the planar Lebesgue measure,
a frequently used notion of smallness of subsets of $\partial \mathbb{D}$ is that of having
logarithmic capacity zero; see \cite[Chapter II]{EKM+14} for an introduction.
This notion is particularly important in the potential theoretic approach to the Dirichlet space.
We will review the definition in Section \ref{sec:prelim}.
A theorem of Beurling shows that every function in $\mathcal{D}$ has radial limits outside
of a subset of $\partial \mathbb{D}$ of (outer) logarithmic capacity zero; see \cite[Section 3.2]{EKM+14}.
In the context of boundary interpolation, Peller and Khrushch\"{e}v \cite{PK82} showed
that a compact set $E \subset \partial \mathbb{D}$ is an interpolation set
for $A(\mathbb{D}) \cap \mathcal{D}$ if and only if $E$ has logarithmic capacity zero.
Many of these considerations have been extended to standard weighted Dirichlet spaces
and their associated capacities, and more generally to Hardy--Sobolev spaces
on the Euclidean unit ball $\mathbb{B}_d$ of $\mathbb{C}^d$ by Cohn \cite{Cohn89}
and by Cohn and Verbitsky \cite{CV95}.

The Hardy space $H^2$, the Dirichlet space $\mathcal{D}$ and more generally
Hardy--Sobolev spaces on the ball belong to a large class
of reproducing kernel Hilbert spaces of holomorphic functions on the ball,
called regular unitarily invariant spaces. We will recall the precise definition
in Section \ref{sec:prelim}.

In studying regular unitarily invariant spaces $\mathcal{H}$ and especially their multipliers,
a functional analytic smallness condition of subsets of $\partial \mathbb{B}_d$
has proved to be very useful in recent years.
This smallness condition has its roots in the study of the ball algebra
\[
A(\mathbb B_d) = \{ f \in C( \overline{\mathbb B_d}) : f \big|_{\mathbb B_d} \in \mathcal O(\mathbb B_d) \}
\]
as explained in Rudin's book \cite[Chapter 10]{Rudin08}.
We let
\begin{equation*}
  \Mult(\mathcal{H}) = \{ \varphi: \mathbb{B}_d \to \mathbb{C}: \varphi \cdot f \in \mathcal{H}
  \text{ whenever } f \in \mathcal{H} \}
\end{equation*}
denote the multiplier algebra of $\mathcal{H}$.
If $\varphi \in \Mult(\mathcal{H})$, its multiplier norm $\|\varphi\|_{\Mult(\mathcal{H})}$
is the norm of the multiplication operator $f \mapsto \varphi \cdot f$ on $\mathcal{H}$.
A complex regular Borel measure $\mu$ on $\partial \mathbb{B}_d$ is said to be \emph{$\Mult(\mathcal{H})$-Henkin}
if whenever $(p_n)$ is a sequence of polynomials satisfying $\|p_n\|_{\Mult(\mathcal{H})} \le 1$ for all $n \in \mathbb{N}$
and $\lim_{n \to \infty} p_n(z) = 0$ for all $z \in \mathbb{B}_d$, we have
\begin{equation*}
  \lim_{n \to \infty} \int_{\partial \mathbb{B}_d} p_n \, d \mu = 0.
\end{equation*}
A Borel subset $E \subset \partial \mathbb{B}_d$ is said to be \emph{$\Mult(\mathcal{H})$-totally null}
if $|\mu|(E) = 0$ for all $\Mult(\mathcal{H})$-Henkin measures $\mu$.
The Henkin condition can be rephrased in terms of a weak-$*$ continuity property;
see Section \ref{sec:prelim} for this reformulation and for more background.

In the case of $H^2$, the F.\ and M.\ Riesz theorem implies that a measure is Henkin if and only if it is absolutely continuous
with respect to Lebesgue measure on $\partial \mathbb{D}$. Hence the totally
null sets are simply the sets of Lebesgue measure $0$.
Beyond the ball algebra, these notions were first studied
by Clou\^atre and Davidson for the Drury--Arveson space \cite{CD16}, and then for more
general regular unitarily invariant spaces by Bickel, M\textsuperscript{c}Carthy and the second named author \cite{BHM17}.
Just as in the case of the ball algebra, Henkin measures and totally null sets appear naturally
when studying the dual space of algebras of multipliers \cite{CD16,DH20}, ideals of multipliers \cite{CD18,DH20},
functional calculi \cite{BHM17,CD16a}, and peak interpolation problems for multipliers \cite{CD16,DH20}.

\subsection{Main results}
In this article, we will compare the functional
analytic notion of being totally null with the potential theoretic notion of having capacity zero.
As was pointed out in \cite{DH20}, for the Dirichlet space $\mathcal{D}$,
the energy characterization of logarithmic capacity easily implies that every compact subset of $\partial \mathbb{D}$ that is $\Mult(\mathcal{D})$-totally
null necessarily has logarithmic capacity zero.
We will show that for Hardy--Sobolev spaces on the ball, including the Dirichlet space on the disc, the two notions of smallness in fact agree.

To state the result, let us recall the definition of Hardy--Sobolev spaces
(a.k.a.\ Besov--Sobolev spaces) on the ball.
Let $\sigma$ denote the normalized surface measure on $\partial \bB_d$
and let
\[
\Hardys = \Big\{f \in \cO(\bB_d) : \sup_{0 \le r < 1} \int_{\partial \bB_d} |f( r \zeta)|^2 d \sigma(\zeta) < \infty \Big\}
\]
be the Hardy space on the unit ball. Let $s \in \bR$.
If $f \in \mathcal{O}(\mathbb{B}_d)$ has a homogeneous decomposition $ f= \sum_{n=0}^\infty f_n$, we let
\begin{equation*}
  \|f\|_s^2 = \sum_{n=0}^\infty (n+1)^{2 s} \|f_n\|^2_{\Hardys}
\end{equation*}
and define
\begin{equation*}
  \mathcal{H}_s = \{ f \in \mathcal{O}(\mathbb{B}_d): \|f\|_s < \infty \}.
\end{equation*}
Thus, if $R^s f = \sum_{n=1}^\infty n^s f_n$ denotes the fractional radial derivative, then
\begin{equation*}
  \mathcal{H}_s = \{f \in \mathcal{O}(\mathbb{B}_d): R^s f \in \Hardys \}.
\end{equation*}
There are also natural $L^p$-versions of these spaces, but we will exclusively work in the Hilbert space setting.
If $s < 0$, then $\mathcal{H}_s$ is a weighted Bergman space on the ball,
and clearly $\mathcal{H}_0 = \Hardys$.
If $d=1$, then $\mathcal{H}_{1/2} = \mathcal{D}$, the classical Dirichlet space on the disc, and if $0 < s < \frac{1}{2}$,
then the spaces $\mathcal{H}_s$ are the standard weighted Dirichlet spaces on the disc.
For $d \ge 1$, the space  $\mathcal{H}_{d/2}$ is sometimes called the Dirichlet space on the ball,
and the spaces $\mathcal{H}_{s}$ for $ \frac{d-1}{2} < s < \frac{d}{2}$ are multivariable versions
of the standard weighted Dirichlet spaces on the disc.
If $s > \frac{d}{2}$, then every function in $\mathcal{H}_s$ extends to a continuous function on $\ol{\mathbb{B}_d}$.

In the range $s \le \frac{d}{2}$, there is a different description of $\mathcal{H}_s$,
which is also sometimes used as the definition.
For $a > 0$, let
\begin{equation*}
  K_a(z,w) = \frac{1}{(1 - \langle z,w \rangle)^a }
\end{equation*}
and
\begin{equation*}
  K_0(z,w) = \log \Big( \frac{e}{1 - \langle z,w \rangle} \Big),
\end{equation*}
and let $\mathcal{D}_a$ denote the reproducing kernel Hilbert space on $\mathbb{B}_d$
with kernel $K_a$.
It is well known that if $a = d - 2s$, then $\mathcal{D}_a = \mathcal{H}_s$, with equivalence of norms.
(This follows by expanding $K_a$ in a power series and comparing the coefficients with
$\|z_1^n\|_{\mathcal H_s}^{-2}$ with the help of Stirlings' formula and the known formula
for $\|z_1^n\|_{\Hardys}$; see \cite[Proposition 1.4.9]{Rudin08}).
The space $\mathcal{D}_1$ is usually called the \emph{Drury--Arveson space $H^2_d$} and plays
a key role in multivariable operator theory \cite{Arveson98} and in the
theory of complete Pick spaces \cite{AM02}.
For more background on these spaces, we refer the reader to \cite{ZZ08}.

For each of the spaces $\mathcal{H}_s$ for $\frac{d-1}{2} < s \le \frac{d}{2}$,
there is a natural notion of (non-isotropic Bessel) capacity $C_{s,2}(\cdot)$,
introduced by Ahern and Cohn \cite{AC89}.
Equivalently, for the spaces $\mathcal D_a$, there is a notion of capacity
that can be defined in terms of the reproducing kernel of $\mathcal D_a$.
We will review these definitions and show their equivalence in Section \ref{sec:prelim}.

Our main result concerning the Hardy--Sobolev spaces $\mathcal H_s$ is the following.

\begin{thm}
  \label{thm:main_dirichlet}
  Let $d \in \mathbb{N}$ and let $\frac{d-1}{2} < s \le \frac{d}{2}$.
  A compact subset $E \subset \partial \mathbb{B}_d$ is $\Mult(\mathcal{H}_s)$-totally null
  if and only if $C_{s,2}(E) = 0$.
\end{thm}

In particular, taking $d=1$ and $s=\frac{1}{2}$, we see that in the context of the classical
Dirichlet space $\mathcal D$, a compact subset $E \subset \partial \mathbb D$ is
$\Mult(\mathcal D)$-totally null if and only if it has logarithmic capacity zero.

A direct proof of Theorem \ref{thm:main_dirichlet} will be provided in Section \ref{sec:Dirichlet}.

Moreover, we will prove an abstract result about totally null sets,
which, in combination with work on exceptional sets by Ahern and Cohn \cite{AC89} and by Cohn and Verbitsky \cite{CV95},
will yield a second proof of Theorem \ref{thm:main_dirichlet}.
This result applies to some spaces that are not covered by Theorem \ref{thm:main_dirichlet}, such as the Drury--Arveson space.

It is possible to interpret the capacity zero condition as a condition involving the reproducing
kernel Hilbert space $\mathcal{H}$ (cf.\ Proposition \ref{prop:energy_fa} below), whereas the totally null condition is a condition
on the multiplier algebra $\Mult(\mathcal{H})$.
Complete Pick spaces form a class of spaces in which it is frequently possible
to go back and forth between $\mathcal{H}$ and $\Mult(\mathcal{H})$;
see the book \cite{AM02} and Section \ref{sec:prelim} of the present article for more background.
For now, let us simply mention that the spaces $\mathcal{D}_a$ for $0 \le a \le 1$ are complete Pick spaces.
(For $a=0$, one needs to pass to a suitable equivalent norm.)

If $\cH$ is a reproducing kernel Hilbert space on $\mathbb B_d$, let us say that a compact subset $E \subset \partial \mathbb{B}_d$ is an \emph{unboundedness set for $\mathcal{H}$} 
if there exists $f \in \mathcal{H}$ so that $\lim_{r \nearrow 1} |f(r \zeta)| = \infty$ for all $\zeta \in E$.
The following result covers the spaces in Theorem \ref{thm:main_dirichlet}, but it also
applies, for example, to the Drury--Arveson space, which corresponds to the end point $s = \frac{d-1}{2}$.

\begin{thm}
  \label{thm:main_CNP}
  Let $\mathcal{H}$ be a regular unitarily invariant complete Pick space on $\mathbb{B}_d$.
  A compact set $E \subset \partial \mathbb{B}_d$ is an unboundedness set for $\mathcal{H}$
  if and only if $E$ is $\Mult(\mathcal{H})$-totally null.
\end{thm}

A refinement of this result will be proved in Section \ref{sec:main_CNP}.
The results of Ahern and Cohn \cite{AC89} and of Cohn and Verbitsky \cite{CV95} on exceptional
sets show that in the case of the spaces $\mathcal{H}_s$ for $\frac{d-1}{2} < s \le \frac{d}{2}$,
a compact subset $E \subset \partial \mathbb B_d$ is an unboundedness set for $\mathcal H_s$
if and only if $C_{s,2}(E) = 0$. Indeed,
the ``only if'' part follows from \cite[Theorem B]{AC89},
the ``if'' part is contained in the construction on p.\ 443 of \cite{AC89}; see also \cite[p. 94]{CV95}.
Thus, we obtain another proof of Theorem \ref{thm:main_dirichlet}.

\subsection{Applications}
We close the introduction by mentioning some applications of Theorem \ref{thm:main_dirichlet}.
The first application concerns peak interpolation.
Extending the work of Peller and Khrushch\"{e}v \cite{PK82} on boundary interpolation in the Dirichlet space,
Cohn and Verbitsky \cite[Theorem 3]{CV95} showed that every compact subset $E \subset \partial \mathbb{B}_d$
with $C_{s,2}(E) = 0$ is a strong boundary interpolation set for $\mathcal{H}_s \cap A(\mathbb{B}_d)$.
This means that for every $g \in C(E)$, there exists $f \in \mathcal{H}_s \cap A(\mathbb{B}_d)$ with
$f \big|_E = g$ and $\max( \|f\|_{\mathcal{H}_s}, \|f\|_{A(\mathbb{B}_d)}) \le \|f\|_{C(E)}$.
Combining Theorem \ref{thm:main_dirichlet} with a peak interpolation result for totally
null sets of Davidson and the second named author \cite{DH20}, we can strengthen the result of Cohn and Verbitsky
in two ways. Firstly, we replace $\mathcal{H}_s \cap A(\mathbb{B}_d)$ with the smaller space $A(\mathcal{H}_s)$,
which is defined to be the multiplier norm closure of the polynomials in $\Mult(\mathcal{H}_s)$. Thus,
\begin{equation*}
  A(\mathcal{H}_s) \subset \Mult(\mathcal{H}_s) \cap A(\mathbb{B}_d) \subset \mathcal{H}_s \cap A(\mathbb{B}_d)
\end{equation*}
with contractive inclusions. Secondly, we obtain a strict pointwise inequality off of $E$.

\begin{thm}
  \label{thm:peak_interpolation}
  Let $d \in \mathbb{N}$, let $\frac{d-1}{2} < s \le \frac{d}{2}$ and let $E \subset \partial \mathbb{B}_d$ be compact with $C_{s,2}(E) = 0$. Then for each $g \in C(E) \setminus \{0\}$, there exists
  $f \in A(\mathcal{H}_s)$ so that
  \begin{enumerate}
    \item $f \big|_E = g$,
    \item $|f(z)| < \|g\|_\infty$ for every $z \in \overline{\mathbb{B}_d} \setminus E$, and
    \item $\|f\|_{\Mult(\mathcal{H}_s)} = \|g\|_\infty$.
  \end{enumerate}
\end{thm}

\begin{proof}
  According to \cite[Theorem 1.4]{DH20}, the conclusion holds when $\mathcal{H}_s$ is replaced with any
  regular unitarily invariant space $\mathcal{H}$ and $E$ is $\Mult(\mathcal{H})$-totally null.
  Combined with Theorem \ref{thm:main_dirichlet}, the result follows.
\end{proof}

In fact, in the setting of Theorem \ref{thm:peak_interpolation},
there exists an isometric linear operator $L: C(E) \to A(\mathcal H_s)$ of
peak interpolation; see \cite[Theorem 8.3]{DH20}.
In a similar fashion, one can now apply other results of \cite{DH20}
in the context of the spaces $\mathcal{H}_s$,
replacing the totally null condition with the capacity zero condition.
In particular, this yields a joint Pick and peak interpolation result (cf.\ \cite[Theorem 1.5]{DH20})
and a result about boundary interpolation in the context of interpolation sequences (cf. \cite[Theorem 6.6]{DH20}).

Our second application concerns cyclic functions. Recall that a function $f \in \mathcal{H}_s$ is said to be \emph{cyclic}
if the space of polynomial multiples of $f$ in dense in $\mathcal{H}_s$.
It is a theorem of Brown and Cohn \cite{BC85} that if $E \subset \partial \mathbb{D}$ has logarithmic
capacity zero, then there exists a function $f \in \mathcal{D} \cap A(\mathbb{D})$ that is cyclic
for $\mathcal{D}$ so that $f \big|_E = 0$; see also \cite{EL19} for an extension to other Dirichlet type
spaces on the disc.
The following result extends the theorem of Brown and Cohn to the spaces $\mathcal{H}_s$ on the ball,
and moreover achieves that $f \in A(\mathcal{H}_s)$, so in particular, $f$ is a multiplier.

\begin{cor}
  Let $d \in \mathbb{N}$, let $\frac{d-1}{2} < s \le \frac{d}{2}$ and let $E \subset \partial \mathbb{B}_d$
  be compact with $C_{s,2}(E)= 0$. Then there exists $f \in A(\mathcal{H}_s)$
  that is cyclic for $\mathcal H_s$
  so that $E = \{z \in \overline{\mathbb{B}_d}: f(z) = 0 \}$.
\end{cor}

\begin{proof}
  Applying Theorem \ref{thm:peak_interpolation} to the constant function $g = 1$, we find $h \in A(\mathcal{H}_s)$
  so that $h \big|_E = 1, |h(z)| < 1$ for $z \in \overline{\mathbb{B}_d} \setminus E$ and $\|h\|_{\Mult(\mathcal{H}_s)} = 1$.
  Set $f = 1 - h$. Clearly, $f$ vanishes precisely on $E$. The fact that $\|h\|_{\Mult(\mathcal{H}_s)} \le 1$ easily
  implies that $f$ is cyclic; see for instance \cite[Lemma 2.3]{AHM+17a} and its proof.
\end{proof}

\section{Preliminaries}
\label{sec:prelim}

\subsection{Regular unitarily invariant spaces and totally null sets}

Throughout, let $d \in \mathbb N$.
A \emph{regular unitarily invariant space} is a reproducing kernel Hilbert space $\mathcal{H}$
on $\mathbb{B}_d$ whose reproducing kernel is of the form
\begin{equation}
\label{eqn:unit_inv_kernel}
  K(z,w) = \sum_{n=0}^\infty a_n \langle z,w \rangle^n,
\end{equation}
where $a_0 = 1$, $a_n > 0$ for all $n \in \mathbb{N}$ and $\lim_{n \to \infty} \frac{a_n}{a_{n+1}} = 1$.
We think of the last condition as a regularity condition, as it is natural
to assume that the power series defining $K$ has radius of convergence $1$, since $\mathcal{H}$
is a space of functions on the ball of radius $1$. Under this assumption, the limit, if it exists,
necessarily equals $1$. We recover $H^2$ and $\mathcal{D}$ by choosing $d=1$
and $a_n = 1$, respectively $a_n = \frac{1}{n+1}$, for all $n \in \mathbb{N}$.
Expanding the reproducing kernels of $\cD_a$ into a power series, one easily checks that $\cD_a$ is a regular unitarily
invariant space for all $a \ge 0$.
While the class of regular unitarily invariant spaces is not stable under passing to an equivalent norm,
one can also check that the spaces $\cH_s$ are regular unitarily invariant spaces for all $s \in \bR$.
Indeed, each space $\mathcal H_s$ has a reproducing kernel as in \eqref{eqn:unit_inv_kernel}, where
\begin{equation*}
    a_n = \|z_1^n\|_{\mathcal H_s}^{-2}.
\end{equation*}
More background on these spaces can be found in \cite{BHM17,DH20,GHX04}.

Let $\mathcal{H}$ be a regular unitarily invariant space. We let $\Mult(\mathcal{H})$ denote the multiplier algebra of $\mathcal{H}$. Identifying a multiplier $\varphi$ with the corresponding multiplication operator on $\mathcal{H}$,
we can regard $\Mult(\mathcal{H})$ as a WOT closed subalgebra of $\mathcal{B}(\mathcal{H})$, the algebra
of all bounded linear operators on $\mathcal{H}$. By trace duality, $\Mult(\mathcal{H})$ becomes
a dual space in this way, and hence is equipped with a weak-$*$ topology.
The density of the linear span of kernel functions in $\mathcal{H}$ implies that on bounded
subsets of $\Mult(\mathcal{H})$, the weak-$*$ topology agrees with the topology
of pointwise convergence on $\mathbb{B}_d$.
In a few places, we will use the following basic and well known fact, which we state as a lemma for easier reference.
For a proof, see for instance \cite[Lemma 2.2]{DH20}.

\begin{lem}
  \label{lem:dilations_convergence}
  Let $\mathcal{H}$ be a regular unitarily invariant space and let $\varphi \in \Mult(\mathcal{H})$.
  Let $\varphi_r(z) = \varphi(r z)$ for $0 \le r \le 1$ and $z \in \mathbb{B}_d$.
  Then $\|\varphi_r\|_{\Mult(\mathcal{H})} \le \|\varphi\|_{\Mult(\mathcal{H})}$
  for all $0 \le r \le 1$ and $\lim_{r \nearrow 1} \varphi_r = \varphi$ in the weak-$*$ topology
  of $\Mult(\mathcal{H})$.
\end{lem}

Let $M(\partial \mathbb{B}_d)$ be the space of complex regular Borel measures on $\partial \mathbb{B}_d$.

\begin{defn}
\label{defn:Henkin_TN}
  Let $\mathcal{H}$ be a regular unitarily invariant space.
  \begin{enumerate}[label=\normalfont{(\alph*)}]
      \item A measure $\mu \in M( \partial \mathbb{B}_d)$ is said to be $\Mult(\mathcal{H})$-Henkin if the functional
\begin{equation*}
  \Mult(\mathcal{H}) \supset \mathbb{C}[z_1,\ldots,z_d] \to \mathbb{C},
  \quad p \mapsto \int_{\partial \mathbb{B}_d} p \, d \mu,
\end{equation*}
extends to a weak-$*$ continuous functional on $\Mult(\mathcal{H})$.
\item A Borel subset $E \subset \partial \mathbb{B}_d$ is said to be $\Mult(\mathcal{H})$-totally null
if $|\mu|(E) = 0$ for all $\Mult(\mathcal{H})$-Henkin measures $\mu$.
  \end{enumerate}
\end{defn}

By \cite[Lemma 3.1]{BHM17}, the definition of Henkin measure given here
is equivalent to the one given in the introduction in terms of sequences of polynomials converging pointwise to zero.
The set of $\Mult(\mathcal{H})$-Henkin measures forms a band (see \cite[Lemma 3.3]{BHM17}), meaning in particular
that $\mu$ is $\Mult(\mathcal{H})$-Henkin if and only if $|\mu|$ is Henkin.
This band property implies that a compact set $E$ is $\Mult(\mathcal{H})$-totally
null if and only if $\mu(E) = 0$ for every positive $\Mult(\mathcal{H})$-Henkin measure
$\mu$ that is supported on $E$; see \cite[Lemma 2.5]{DH20}.

Finally, we require the notion of a \emph{complete Pick space}.
Complete Pick spaces are reproducing kernel Hilbert spaces that are
defined in terms of an interpolation condition for multipliers; see the book \cite{AM02}
for more background.
In the context of regular unitarily invariant spaces, there is a concrete
characterization in terms of the reproducing kernel. If the reproducing
kernel of $\mathcal{H}$ is $K(z,w) = \sum_{n=0}^\infty a_n \langle z,w \rangle^n$,
then $\mathcal{H}$ is a complete Pick space if and only if the sequence $(b_n)_{n=1}^\infty$
defined by the power series identity
\begin{equation*}
  \sum_{n=1}^\infty b_n t^n = 1 - \frac{1}{\sum_{n=0}^\infty a_n t^n}
\end{equation*}
satisfies $b_n \ge 0$ for all $n \in \mathbb{N}$ (this is a straightforward generalization of \cite[Theorem 7.33]{AM02}). In particular, the spaces $\mathcal{D}_a$ are complete Pick spaces in the range $0 \le a \le 1$ ; cf.\ \cite[Lemma 7.38]{AM02}. (For $a=0$, one needs to pass for instance to the equivalent norm induced by the reproducing
kernel $\frac{1}{\langle z, w \rangle} \log\big( \frac{1}{1 - \langle z,w \rangle} \big)$.)

\subsection{Capacity}

When defining capacity for (compact) sets $E \subset \partial \bB_d$ induced by the  Hardy-Sobolev spaces $\cH_s$, there are at least two possible approaches.
Each one can be viewed as natural, depending on the perspective,
and in fact we require both approaches in the proof of Theorem \ref{thm:main_dirichlet}, one for each implication.
The two definitions turn out to be equivalent in the sense that the capacities defined are comparable with absolute constants. In particular, capacity zero sets coincide in both senses. We shall briefly discuss the equivalence of the two definitions.

The first definition, introduced in \cite[p.\ 489]{AC89}, is motivated by the fact that the spaces $\cH_s$ can be understood as potential spaces and
fits into the framework of the general potential theory of Adams and Hedberg \cite{Adams1996}.
Let $M^+(\partial \mathbb{B}_d)$ denote the set of positive regular Borel measures on $\partial \mathbb{B}_d$.
We let $\sigma$ be the normalized surface measure on $\partial \mathbb B_d$.
If $E \subset \partial \mathbb B_d$ is compact, let $M^+(E)$ be the set of all
measures in $M^+(\partial \mathbb B_d)$ that are supported on $E$.
For $0 \le s < d$, consider the kernel
\[
k_s(z,w) = \frac{1}{|1 - \langle z,w \rangle|^{d-s}} \quad (z,w \in \overline{\bB_d})
\]
and also set
\[
k_d(z,w) = \log \frac{e}{|1 - \langle z,w \rangle|} \quad (z,w \in \overline{\bB_d}).
\]

\begin{defn}
Let $0 \le s \le d$, let $\mu \in M^+(\partial \mathbb B_d)$ and let $E \subset \partial \mathbb B_d$ be compact.
  \begin{enumerate}[label=\normalfont{(\alph*)}]
  \item The \emph{non-isotropic Riesz potential} of $\mu$ is
\[ \cI_s(\mu)(z) = 
\int_{\partial \bB_d} k_s(z,w) d \mu(w) \quad (z \in \partial \mathbb B_d).\]
We extend the definition to non-negative measurable functions $f \in L^1(\partial \bB_d, d\sigma ) $ by letting 
$\cI_s(f) = \cI_s(f \, d\sigma )$.
\item The \emph{non-isotropic Bessel capacity} of $E$ is defined by
\[ C_{s,2}(E)= \inf \{ \norm f ^2_{L^2(\partial \bB_d , d\sigma )} : \cI_s(f) \geq 1 \,\,\, \text{on} \,\,\, E, f \geq 0 \}. \]
\item The quantity $\|\cI_s(\mu)\|_{L^2(\partial \bB_d, d \sigma)}^2 \in [0,\infty]$ is called the \emph{energy} of $\mu$.
  \end{enumerate}
\end{defn}

By \cite[Theorem 2.5.1]{Adams1996}, we have the following ``dual'' expression for the capacity $C_{s,2}(\cdot)$, 
\begin{equation}
\label{eqn:dual}
     C_{s,2}(E)^{1/2} = \sup \{ \mu(E): \mu \in M^+(E), \,\,\, \norm { \cI_s(\mu) }_{L^2(\partial \bB_d, d\sigma)} \leq 1  \}.
\end{equation}
In particular, $C_{s,2}(E) > 0$ if and only if $E$ supports a probability measure of finite energy.

A different approach, which can be justified by regarding $\cH_s$ as a reproducing kernel Hilbert space, is the following; cf. \cite[Chapter 2]{EKM+14}.
Recall that if $a=d-2s$, then $\mathcal H_s = \mathcal D_a$ with equivalent norms. Moreover, we have $k_{2s} = |K_a|$.

\begin{defn}
    Let $\frac{d-1}{2} < s \le \frac{d}{2}$, let $a= d- 2 s$, let $\mu \in M^+(\partial \mathbb{B}_d)$ and let $E \subset \partial \mathbb{B}_d$ be compact.
  \begin{enumerate}[label=\normalfont{(\alph*)}]
    \item The $\mathcal D_a$-\emph{potential} of $\mu$ is
      \begin{equation*}
        \cI_{2s}(\mu)(z) = \int_{\partial \mathbb B_d} |K_a(z,w)| \, d \mu  (w).
      \end{equation*}
    \item The $\mathcal D_a$-\emph{energy}  of $\mu$ is defined by
      \begin{equation*}
        \mathcal{E}(\mu, \mathcal D_a)
        = \int_{\partial \bB_d} \int_{\partial \bB_d} |K_a(z,w)| d \mu(z) d \mu(w).
      \end{equation*}
    \item The $\mathcal D_a$-\emph{capacity} of $E$ is defined by
      \begin{equation*}
        \capac{E}{\mathcal D_a}^{1/2} =  \sup \{ \mu(E): \mu \in M^+(E) , \,\, \mathcal{E}(\mu, \mathcal D_a) \leq 1 \}.
      \end{equation*}
  \end{enumerate}
\end{defn}

As for the Bessel capacity, $\capac{E}{\mathcal{D}_a}> 0$ if and only if $E$ supports a probability measure of finite energy.
The formula $d(z,w) = | 1 - \langle z, w \rangle|^{1/2}$ defines a metric on $\partial \mathbb B_d$,
called the \emph{Koranyi metric}; see \cite[Proposition 5.1.2]{Rudin08}.
Thus, the capacities $\capac{\cdot}{\mathcal D_a}$
fit into the framework of capacities on compact metric spaces developed in \cite[Chapter 2]{EKM+14}.

It appears to be well known to experts that the capacities $\capac{\cdot}{\mathcal D_a}$ and $C_{s,2}(\cdot)$ are equivalent
if $a = d- 2s$,
the point being that the corresponding energies are comparable.
A proof in the case $d=1, s=\frac{1}{2}$ can be found in \cite[Lemma 2.2]{Cascante2012}.
In the case $s \neq \frac{d}{2}$, the crucial estimate is stated in \cite[Remark 2.1]{CO95} without proof.
A proof of the estimate in one direction in this case is contained in \cite[p.\ 442-442]{AC89}.
For the sake of completeness, we provide an argument that applies to all cases under consideration.
We adapt the proof in \cite[Lemma 2.2]{Cascante2012} to the non-isotropic geometry of $\partial \bB_d$.

Throughout, we write $A \lesssim B$ to mean that there exists a constant $C \in (0,\infty)$ so that $A \le C B$,
and $A \approx B$ to mean that $A \lesssim B$ and $B \lesssim A$.

\begin{lem}\label{lem:equiv_riesz_pot}
  Let $\frac{d-1}{2} < s \leq \frac{d}{2}$ and $\mu \in M^+(\partial \bB_d ) $. Then
  \[  \cI_s(\cI_s(\mu)) \approx \cI_{2s}(\mu), \]
  where the implied constants only depend on $s$ and $d$.
\end{lem}

\begin{proof}

We will show that
\[
\int_{\partial \bB_d} k_s(z, \zeta) k_s(\zeta,w) d \sigma(\zeta) \approx k_{2s}(z,w) \quad (z,w \in \partial \bB_d).
\]
The statement then follows by integrating both sides with respect to $\mu$ and using Fubini's theorem.

Let
\[ d(z,w)=|1-\inner{z}{w}|^{\frac{1}{2}} \quad (z,w \in \partial \mathbb B_d) \]
be the Koranyi metric; see \cite[Proposition 5.1.2]{Rudin08}.
If $\delta>0$ we let $Q_\delta(z) = \{ \zeta \in \partial \bB_d: d(\zeta,z) \le \delta\}$ be the Koranyi ball centered at $z$ with radius $\delta.$
We will repeatedly use the fact that $\sigma(Q_{\delta}(z)) \approx \delta^{2 d}$ for $0 \le \delta \le \sqrt{2}$; see \cite[Proposition 5.1.4]{Rudin08}.

Now let $z,w\in \partial \bB_d$ and set $\delta=\frac{d(z,w)}{2}$. Then, in order to estimate the kernel
\[ \int_{\partial   \bB_d} \frac{d\sigma(\zeta)}{|1-\inner{z}{\zeta}|^{d-s}|1-\inner{\zeta}{w}|^{d-s}}
= \int_{\partial \bB_d} \frac{d \sigma(\zeta)}{d(z,\zeta)^{2(d-s)} d(\zeta,w)^{2(d-s)}},
\]
we split the domain of integration $\partial \bB_d$ as follows
\[ \partial \bB_d = Q_\delta(z) \cup Q_\delta(w) \cup  \{ d(\zeta,z) \leq d(\zeta,w) \} \setminus Q_\delta(z) \cup  \{ d(\zeta,w) \leq d(\zeta,z) \} \setminus Q_\delta(w).  \]
We denote by \MakeUppercase{\romannumeral 1}, \MakeUppercase{\romannumeral 1}$'$, \MakeUppercase{\romannumeral 2}, \MakeUppercase{\romannumeral 2}$'$ the corresponding integrals. By the symmetry of the problem it suffices to estimate \MakeUppercase{\romannumeral 1} and \MakeUppercase{\romannumeral 2}.

For \MakeUppercase{\romannumeral 1}, we note that if $\zeta \in Q_{\delta}(z)$, then $\delta \le d(\zeta,w) \le 3 \delta$
by the triangle inequality for $d$.
Hence, integrating with the help of the the distribution function, we find that
\begin{align*}
    \text{\MakeUppercase{\romannumeral 1}} & \approx \frac{1}{\delta^{2(d-s)}} \int_{Q_\delta(z)}\frac{d\sigma(\zeta)}{d(z,\zeta)^{2(d-s)}} \\
    &= \frac{1}{\delta^{2(d-s)}} \int_{0}^\infty \sigma( \{ \zeta \in Q_{\delta}(z) : d(z,\zeta) \le t^{\frac{-1}{2 (d-s)}} \}) dt\\
           & = \frac{\sigma ( Q_\delta(z) )}{\delta^{4(d-s)}}   +  \frac{1}{\delta^{2(d-s)}} \int_{\delta^{-2(d-s)}}^\infty \sigma ( \{ \zeta \in \partial \mathbb B_d:  d(z,\zeta) \leq t^{\frac{-1}{2(d-s)}} \} ) dt \\
           & \approx \delta^{-2(d-2s)} + \frac{1}{\delta^{2(d-s)}} \int_{\delta^{-2(d-s)}}^{\infty} t^{\frac{-d}{d-s}}dt \\
           & \approx \delta^{-2(d-2s)}.
\end{align*}

Next, using the fact that $d(z,\zeta) \le \sqrt{2}$ for all $z, \zeta \in \partial \bB_d$, we see that
\begin{align*}
    \textup{II} &\le \int_{\partial \bB_d \setminus Q_{\delta}(z)} \frac{d \sigma(\zeta)}{d(z,\zeta)^{4 (d-s)}} \\
     &= \int_{0}^\infty \sigma( \{ \zeta \in \partial \bB_d : \delta < d(z,\zeta) \le t^{\frac{-1}{4(d-s)}} \} ) dt\\
     &\lesssim 1 + \int_{2^{-2(d-s)}}^{\delta^{-4 (d-s)}} t^{\frac{-d}{2(d-s)}} dt \\
     &\lesssim
     \begin{cases}
        \delta^{-2(d-2s)}, & \text{ if } s < \frac{d}{2}, \\
        \log( \delta^{-2}), & \text { if } s = \frac{d}{2}.
     \end{cases}
\end{align*}
Combining the estimates for \textup{I} and \textup{II} and recalling the definition of $\delta$ we see that
\[
\int_{\partial   \bB_d} \frac{d\sigma(\zeta)}{|1-\inner{z}{\zeta}|^{d-s}|1-\inner{\zeta}{w}|^{d-s}}
\lesssim
\begin{cases}
\frac{1}{|1 - \langle z,w \rangle |^{d - 2 s}}, & \text{ if } s < \frac{d}{2} \\
\log\big( \frac{e}{|1 - \langle z,w \rangle|} \big), & \text{ if } s = \frac{d}{2}.
\end{cases}
\]

To establish the lower bound, it suffices to consider $z,w \in \partial \bB_d$ for which $d(z,w)$ is small.
In the case $s < \frac{d}{2}$, the lower bound follows from the treatment of the integral $\textup{I}$ above.
Let $s = \frac{d}{2}$.
Notice that in the region $ \cU_{z,w}  = \{ \zeta \in \partial \bB_d : d(z,w)  \leq d(w,\zeta) \} $,  the triangle inequality yields $d(\zeta,z) \leq 2d(\zeta,w)$.
Hence integrating again with the distribution function and writing $\delta  = d (z,w )$, we estimate
\begin{align*}
    \int_{\partial \bB_d} \frac{d\sigma(\zeta)}{d(\zeta,w)^dd(\zeta,z)^d} & \gtrsim \int_{ \cU_{z,w} } \frac{d\sigma(\zeta)}{d(\zeta,w)^{2d}} \\
    & = \int_0^{\delta^{-{2d}}} \sigma( \{ \zeta \in \partial \bB_d : \delta \leq  d(\zeta,w) \leq t^{\frac{-1}{2d}} \} ) dt \\
    & = \int_{0}^{\delta^{-{2d}}} \sigma ( Q_{t^{-\frac{1}{2d}}}(w) ) dt - \delta^{-{2d}} \sigma( Q_\delta(w) ) \\ 
    &   \geq c_0 \log ( \delta^{-1} ) - c_1,
\end{align*}
where $c_0,c_1 > 0$ are constants depending only on the dimension $d$. This shows the lower bound for small $\delta$, which concludes the proof.
\end{proof}

From this lemma the equivalence of the capacities $C_{s,2}(\cdot)$ and $\capac{\cdot}{\mathcal D_a}$ for $a = d - 2s$ follows easily.
\begin{cor}
\label{cor:energy_comparable}
  Let $\frac{d-1}{2} < s \leq \frac{d}{2}$, let $a = d - 2s$ and $\mu \in M^+(\partial \bB_d ) $. Then
  \[
  \|\cI_s(\mu)\|^2_{L^2(\partial \bB_d, d \sigma)} \approx \cE(\mu,\mathcal D_a).
  \]
  Hence $C_{s,2}(E) \approx \capac{E}{\cD_a}$ for compact subsets $E \subset \partial \bB_d$.
  Here, all implied constants only depend on $d$ and $s$.
\end{cor}

\begin{proof}
For a measure $\mu \in M^+(\partial \bB_d) $, we compute
\begin{align*}
    \norm{\cI_s(\mu)}^2_{L^2(\partial \bB_d, d\sigma)}
    &  = \int_{\partial \bB_d} \Big( \int_{\partial \bB_d} \frac{d\mu(z)}{|1-\inner{z}{\zeta}|^{d-s}} \Big)^2 d\sigma \\
    &  = \int_{\partial \bB_d} \int_{\partial \bB_d}\int_{\partial   \bB_d} \frac{d\sigma(\zeta)}{|1-\inner{z}{\zeta}|^{d-s}|1-\inner{w}{\zeta}|^{d-s}}  d\mu(z)d\mu(w) \\ 
    &  = \int_{\partial \bB_d}  \cI_s(\cI_s(\mu)) d\mu.
\end{align*}
Thus, Lemma \ref{lem:equiv_riesz_pot} yields that
\[
    \norm{\cI_s(\mu)}^2_{L^2(\partial \bB_d, d\sigma)}
     \approx \int_{\partial \bB_d} \cI_{2s}(\mu)(w) d\mu(w)
    = \cE(\mu, \mathcal D_a).
\]
Since the energies involved are comparable, so are the capacities by \eqref{eqn:dual}.
\end{proof}

\section{Direct proof of Theorem \ref{thm:main_dirichlet}}
\label{sec:Dirichlet}

To prove Theorem \ref{thm:main_dirichlet}, we will make use of holomorphic potentials.
Since several of our proofs involve reproducing kernel arguments, it is slightly more convenient
to work with the spaces $\mathcal D_a$ rather than with $\mathcal H_s$.

\begin{defn}
  Let $0 \le a < 1$ and let $\mu \in M^+(\partial \mathbb{B}_d)$.
  The holomorphic potential of $\mu$ is the function
  \begin{equation*}
    f_\mu: \mathbb{B}_d \to \mathbb{C}, \quad z \mapsto \int_{\partial \mathbb{B}_d} K_a(z,w) \, d\mu(w).
  \end{equation*}
\end{defn}

Let $A(\mathbb{B}_d)$ denote the ball algebra.
If $\mu \in M^+(\partial \mathbb{B}_d)$, let
\begin{equation*}
  \rho_\mu: A(\mathbb{B}_d) \to \mathbb{C}, \quad f \mapsto \int_{\partial \mathbb{B}_d} f \, d \mu
\end{equation*}
denote the associated integration functional.

The following functional analytic interpretation of the holomorphic potential
and of capacity will show that every totally null set has capacity zero.
In the case of the Dirichlet space on the disc, it is closely related to the energy formula
for logarithmic capacity in terms of the Fourier coefficients of a measure; see for instance \cite[Theorem 2.4.4]{EKM+14}.

\begin{prop}
  \label{prop:energy_fa}
  Let $\mu \in M^+(\partial \mathbb{B}_d)$ and let $0 \le a < 1$. The following assertions are equivalent:
  \begin{enumerate}[label=\normalfont{(\roman*)}]
    \item $\mathcal{E}(\mu,\mathcal D_a) < \infty$,
    \item the densely defined functional $\rho_\mu$ is bounded on $\mathcal{D}_a$,
    \item $f_\mu \in \mathcal{D}_a$.
  \end{enumerate}
  In this case,
  \begin{equation*}
    \mathcal{E}(\mu,\mathcal D_a) \approx \|\rho_\mu\|^2_{(\mathcal{D}_a)^*} = \| f_\mu\|_{\mathcal{D}_a}^2,
  \end{equation*}
  where the implied constants only depend on $a$ and $d$,
  and
  \begin{equation*}
    \rho_\mu(g) = \langle g, f_\mu \rangle_{\mathcal{D}_a}
  \end{equation*}
  for all $g \in \mathcal{D}_a$.
\end{prop}

\begin{proof}
  For ease of notation, we write $f = f_\mu$, $\rho = \rho_\mu$ and $k_w(z) = K_a(z,w)$.
  For $0 \le r < 1$, define $f_r(z) = f(rz)$ and $\rho_r(f) = \rho(f_r)$.
  Then each $f_r \in \mathcal{D}_a$ and each $\rho_r$ is a bounded functional on $\mathcal{D}_a$.
  First, we connect $f_r$ and $\rho_r$, which will be useful in all parts of the proof.
  By the reproducing property of the kernel, we find that
  \begin{equation*}
    \langle k_z, f_r \rangle = \overline{f(r z)} = \int_{\partial \mathbb{B}_d} k_{rz}(w) \, d \mu(w)
    = \int_{\partial \mathbb{B}_d} k_z(r w) \, d\mu (w)
    = \rho_r(k_z)
  \end{equation*}
  for all $z \in \mathbb{B}_d$. Since finite linear combinations of kernels are dense in $\mathcal{D}_a$,
  it follows that
  \begin{equation}
    \label{eqn:proof_energy_fa}
    \rho_r(g) =  \langle g, f_r \rangle
  \end{equation}
  for all $g \in \mathcal{D}_a$ and hence $\|\rho_r\|_{(\mathcal{D}_a)^*} = \|f_r\|_{\mathcal{D}_a}$.

  Next, we show the equivalence of (ii) and (iii).
  If $f \in \mathcal{D}_a$, then $\lim_{r \nearrow 1} f_r = f$ in $\mathcal{D}_a$ and hence for all $g \in A(\mathbb{B}_d) \cap \mathcal{D}_a$,
  Equation \eqref{eqn:proof_energy_fa} shows that
  \begin{equation*}
    \rho(g) = \lim_{r \nearrow 1} \rho_r(g) = \lim_{r \nearrow 1} \langle g, f_r \rangle =  \langle g,f \rangle,
  \end{equation*}
  so $\rho$ is bounded on $\mathcal{D}_a$. In this case, $\|\rho\|_{(\mathcal{D}_a)^*} = \|f\|_{\mathcal{D}_a}$,
  which establishes the final statement of the proposition as well.
  Conversely, if $\rho$ is bounded on $\mathcal{D}_a$, then since $\|\rho_r\|_{(\mathcal{D}_a)^*}
  \le \|\rho\|_{(\mathcal{D}_a)^*}$, it follows that $\sup_{0 \le r < 1} \|f_r\|_{\mathcal{D}_a} \le \|\rho\|_{(\mathcal{D}_a)^*}$,
  hence $f \in \mathcal{D}_a$.

  It remains to show the equivalence of (i) and (iii) and that $\mathcal{E}(\mu,\mathcal D_a) \approx \|f\|_{\mathcal{D}_a}^2$.
  With the help of Equation \eqref{eqn:proof_energy_fa}, we see that
  \begin{align*}
    \|f_r\|_{\mathcal{D}_a}^2 = \langle f_r, f_r \rangle = \rho_r(f_r)
    &= \int_{\partial \mathbb{B}_d} f(r^2 z) \, d \mu(z) \\
    &= \int_{\partial \mathbb{B}_d} \int_{\partial \mathbb{B}_d} K_a(rz,rw) \, d\mu(w) d \mu(z).
  \end{align*}
  Taking real parts and using the fact that $\Re K_a$ and $|K_a|$
  are comparable, we find that
  \begin{equation*}
    \|f_r\|_{\mathcal{D}_a}^2
    \approx \int_{\partial \mathbb{B}_d} \int_{\partial \mathbb{B}_d} |K_a(r z, rw)| d \mu(z) d \mu(w).
  \end{equation*}
  Thus, if $f \in \mathcal{D}_a$, then Fatou's lemma shows that
  \begin{equation*}
    \mathcal{E}(\mu,\mathcal D_a) = \int_{\partial \mathbb{B}_d} \int_{\partial \mathbb{B}_d} |K_a(z,w)| \, d \mu(z) d \mu(w) \lesssim \|f\|_{\mathcal{D}_a}^2.
  \end{equation*}
  Conversely, if $\mathcal{E}(\mu,\mathcal D_a) < \infty$, we use the basic inequality
  \begin{equation*}
    \Big| \frac{1}{1 - r^2 \langle z,w \rangle} \Big| \le 2 \Big| \frac{1}{ 1 - \langle z,w \rangle} \Big| \quad (z,w \in \mathbb{B}_d)
  \end{equation*}
  and the Lebesgue dominated convergence theorem to find that
  \begin{equation*}
    \lim_{r \nearrow 1} \|f_r\|_{\mathcal{D}_a}^2 \lesssim \mathcal{E}(\mu,\mathcal D_a),
  \end{equation*}
  so $f \in \mathcal{D}_a$ and $\|f\|_{\mathcal{D}}^2 \lesssim \mathcal{E}(\mu,\mathcal D_a)$.
\end{proof}

With this proposition in hand, we can prove the ``only if'' part of Theorem \ref{thm:main_dirichlet},
which we restate in equivalent form (see Corollary \ref{cor:energy_comparable} for the equivalence).
The idea is the same as that in the proof of \cite[Proposition 2.6]{DH20}.

\begin{prop}
\label{prop:only_if}
  Let $0 \le a < 1$ and let $E \subset \partial \mathbb{B}_d$ be compact.
  If $E$ is $\Mult(\mathcal{D}_a)$-totally null, then $\capac{E}{\mathcal{D}_a} = 0$.
\end{prop}

\begin{proof}
Suppose that $\capac{E}{\mathcal{D}_a} > 0$. Then $E$ supports a probability measure $\mu$ of finite energy $\mathcal E(\mu,\mathcal D_a)$.
  By Proposition \ref{prop:energy_fa}, we see that the integration functional $\rho_\mu$ is bounded
  on $\mathcal{D}_a$. In particular, it is weak-$*$ continuous on $\Mult(\mathcal{D}_a)$.
  Hence $E$ is not $\Mult(\mathcal{D}_a)$-totally null.
\end{proof}

To prove the converse, we require the following fundamental properties of the holomorphic potential of a capacitary extremal measure
of a compact subset $E \subset \mathbb B_d$, i.e. a measure for which the supremum in \eqref{eqn:dual} is achieved.
If $a> 0$, these properties are contained in the proof of \cite[Theorem 2.10]{AC89}, see also \cite[Lemma 2.3]{Cascante2012}
for a proof in the case $d=1$ and $a=0$.
An argument that directly works with the capacity $\capac{\cdot}{\mathcal D_0}$ in the case $d=1$ and $a=0$
can be found on pp.\ 40--41 of \cite{EKM+14}.
We briefly sketch the argument in general.
\begin{lem}
  \label{lem:potential_basic}
  Let $E \subset \partial \mathbb{B}_d$ be a compact set with $\capac{E}{\mathcal{D}_a} > 0$.
  There exists a positive measure $\mu$ supported on $E$ so that the corresponding holomorphic
  potential $f_\mu$ satisfies
  \begin{enumerate}[label=\normalfont{(\alph*)}]
    \item $f_\mu \in \mathcal D_a$ with $\|f_\mu\|^2_{\cD_a} \lesssim \capac{E}{\cD_a}$.
    \item $\liminf_{r \nearrow 1} \Re f_\mu(r \zeta) \gtrsim 1$ for all $\zeta \in \inter(E)$, and
    \item $|f_\mu(z)| \lesssim 1$ for all $z \in \mathbb{B}_d$.
  \end{enumerate}
  Here, the implied constants only depend on $a$ and $d$.
\end{lem}

\begin{proof}
  Let $s = \frac{d-a}{2}$, so that $\mathcal H_s = \mathcal D_a$ with equivalent norms.
  The general theory of Bessel capacities (see \cite[Theorem 2.5.3]{Adams1996}), combined with the maximum principle
  for the capacity $C_{s,2}(\cdot)$ \cite[Lemma 1.15]{AC89} implies that there exists a positive measure $\mu$ supported on $E$ so that
  \begin{enumerate}[]
      \item $\mu(E) = \|\mathcal I_s(\mu)\|^2_{L^2(\partial \mathbb B_d, d \sigma)} = C_{s,2}(E)$;
      \item $\mathcal{I}_s(\mathcal I_s(\mu)) \ge 1$ on $E \setminus F$, where
      $F$ is a countable union of compact sets of $C_{s,2}$-capacity zero.
      \item $\mathcal{I}_s(\mathcal I_s(\mu)) \lesssim 1$ on $\partial \mathbb B_d$.
  \end{enumerate}
  (See also \cite[Corollary 2.4.3]{EKM+14} for an approach using $\capac{\cdot}{\mathcal D_0}$ in the case $d=1$ and $a=0$.)
  
  Item (1) and Corollary \ref{cor:energy_comparable} show that $\mathcal E(\mu, \mathcal D_a) \approx \|\mathcal I_{s}(\mu) \|^2_{L^2(\partial \bB_d, \sigma)} \approx \capac{E}{\mathcal D_a}$,
  hence Proposition \ref{prop:energy_fa} yields that (a) holds.
  
  Lemma \ref{lem:equiv_riesz_pot} and Item (3) show that for $z \in \partial \mathbb B_d$, we have
  \[
     \int_{\partial \mathbb B_d} |K_a(z,w)| d \mu(w) = \mathcal I_{2s}(\mu)(z) \approx \mathcal I_s(\mathcal I_s(\mu))(z) \lesssim 1,
  \]
  so in combination with the basic inequality $| \frac{1}{1 - r \langle z, w \rangle}| \le 2 | \frac{1}{1 - \langle z,w \rangle }|$ for $z,w \in \partial \mathbb B_d$
  and $0 \le r <1$, we see that (c) holds.
  
  To establish (b), notice that (c) implies that $f_\mu \in H^\infty(\mathbb B_d)$,
  so $f_\mu$ has radial boundary limits $f_\mu^*$ almost everywhere with respect to $\sigma$, and $f_\mu = P[f_\mu^*]$, the Poisson
  integral of $f_\mu^*$. 
  Fatou's lemma and the fact that $\Re K_a$ and $|K_a|$ are comparable show that for $\sigma$-almost every $z \in \partial \mathbb B_d$,
  the estimate
  \begin{align*}
    \Re f_{\mu}^*(z) = \lim_{r \nearrow 1} \int_{\partial \mathbb B_d} \Re K_a(r z, w) d \mu(w) &\gtrsim \int_{\partial \mathbb B_d} |K_a(z,w)| \, d \mu (w) \\
    &= \mathcal I_{2s}(\mu)(z).
    \end{align*}
    Now $C_{s,2}(K) = 0$ implies that $\sigma(K) = 0$ for compact sets $K \subset \partial \bB_d$.
    (This is because $\sigma \big|_K$ has finite energy, which for instance follows from Proposition \ref{prop:energy_fa} since $\cD_a$ is continuously contained in $\Hardys$.)
    Therefore, Item (2) and Lemma \ref{lem:equiv_riesz_pot} imply that
    $\Re f_{\mu}^*(z) \gtrsim 1$ for $\sigma$-almost every $z \in E$. In combination with $\Re f_\mu = P[ \Re f_{\mu}^*]$,
    this easily implies (b).
\end{proof}

In \cite{Cascante2014}, Cascante, F\`{a}brega and Ortega showed that if $0 < a < 1$ and if the holomorphic potential $f_\mu$ is bounded in $\mathbb{B}_d$, then
it is a multiplier of $\mathcal{D}_a$.
They also proved an $L^p$-analogue of this statement.
We will require an explicit estimate for the multiplier norm of $f_\mu$.
It seems likely that the arguments in \cite{Cascante2014} could be used to obtain such an estimate.
Instead, we will provide a different argument in the Hilbert space setting, based on the following result of Aleman, M\textsuperscript{c}Carthy,
Richter and the second named author \cite{AHM+17c}.
It also applies to the case $a=0$.
The function $V_f$ below is called the \emph{Sarason function}  of $f$.

\begin{thm}[\cite{AHM+17c}]
  \label{thm:Sarason_function}
  Let $0 \le a < 1$, let $f \in \mathcal{D}_a$ and define
  \begin{equation*}
    V_f(z) = 2 \langle  f, K_a(\cdot,z) f \rangle  - \|f\|^2.
  \end{equation*}
  If $\Re V_f$ is bounded in $\mathbb{B}_d$, then $f \in \Mult(\mathcal{D}_a)$ and
  \begin{equation*}
    \|f\|_{\Mult(\mathcal{D}_a)} \lesssim \| \Re V_f\|_{\infty}^{1/2},
  \end{equation*}
  where the implied constant only depends on $a$ and $d$.
\end{thm}

\begin{proof}
  In \cite[Theorem 4.5]{AHM+17c}, it is shown that if $\mathcal{H}$ is a normalized complete Pick space
  that admits an equivalent norm which is given by an $L^2$-norm of derivatives of order at most $N$,
  then boundedness of $\Re V_f$ implies that $f \in \Mult(\mathcal{H})$, and
  \begin{equation}
    \label{eqn:sarason_weak}
    \|f\|_{\Mult(\mathcal{H})} \lesssim (\|\Re V_f\|_\infty + 3)^{N+\frac{1}{2}}.
  \end{equation}
  This applies in particular to the spaces $\mathcal{D}_a$ for $0 < a < 1$.

  For $a=0$, we need to pass to an equivalent norm to obtain a normalized complete Pick space.
  Explicitly, define
  \begin{equation*}
    K(z,w) = \frac{K_0(z,w)+1}{2} = \sum_{n=0}^\infty a_n \langle z,w \rangle^n,
  \end{equation*}
  where $a_0 = 1$ and $a_n = \frac{1}{2n}$ for $n \ge 1$, and let $\widetilde{\mathcal{D}_0}$ be the associated
  reproducing kernel Hilbert space on $\mathbb{B}_d$.
  Since the sequence $(a_n/a_{n+1})$ is decreasing, $\widetilde{\mathcal{D}_0}$ is a normalized complete Pick space (see \cite[Lemma 7.38]{AM02}).
  The inner product in $\widetilde{\mathcal{D}_0}$ is given by
  \begin{equation*}
    \langle f,g \rangle_{\widetilde{\mathcal{D}_0}} = 2 \langle f,g \rangle_{\mathcal D_0} - f(0) \overline{g(0)};
  \end{equation*}
  in particular, $\widetilde{\mathcal{D}_0}= \mathcal{D}_0$ with equivalence of norms.
  Moreover,
  \begin{equation*}
    \langle f, K(\cdot,z) f \rangle_{\widetilde{\mathcal{D}_0}}
    = \langle f, K_0(\cdot,z) f \rangle_{\mathcal{D}_0} + \|f\|_{\mathcal{D}_0}^2 - |f(0)|^2,
  \end{equation*}
  so if $V_f^{\mathcal{H}}$ denotes the Sarason function of $f$ as determined by $\mathcal{H}$, then
  since $V_f^{\mathcal{D}_0}(0) = \|f\|_{\mathcal{D}_0}^2$, we see that
  \begin{equation*}
    \sup_{z \in \mathbb{B}_d} |\Re V_f^{\widetilde{\mathcal{D}_0}}(z)| 
    \lesssim \sup_{z \in \mathbb{B}_d} | \Re V_f^{\mathcal{D}_0}(z)|.
  \end{equation*}
  Consequently,  \eqref{eqn:sarason_weak} applies to $\mathcal{H} = \mathcal{D}_0$ as well.
  
  The improved bound on the multiplier norm of $f$ follows from the scaling properties of both sides
  of the inequality \eqref{eqn:sarason_weak}. Indeed, if $t > 0$, then $V_{t f} = t^2 V_f$, so applying inequality \eqref{eqn:sarason_weak}
  to the function $t f$, we find that
  \begin{equation*}
    \|f\|_{\Mult(\mathcal{D}_a)}^2 \lesssim \frac{1}{t^2} (t^2 \|\Re V_f\|_\infty +3)^{2 N + 1}
  \end{equation*}
  for all $t > 0$.
  If $\|\Re V_f\|_\infty = 0$, then taking $t \to \infty$ above yields $f=0$.
  If $\|\Re V_f\|_\infty \neq 0$, then choosing $t = \| \Re V_f\|_{\infty}^{-1/2}$,
  we obtain the desired estimate.
  (The choice of $t$ could be optimized to improve the implicit constants, but we will not do so here.)
\end{proof}

With the help of Theorem \ref{thm:Sarason_function}, we can establish the
desired multiplier norm estimate of $f_\mu$.
It can be regarded as a quantitative version of the result of Cascante, F\`abrega and Ortega \cite{Cascante2014}
in the Hilbert space setting.

\begin{prop}
  \label{prop:multiplier_norm}
  Let $0 \le a < 1$ and let $\mu \in M^+(\partial \mathbb{B}_d)$. If $f_\mu$ is bounded in $\mathbb{B}_d$,
  then $f_\mu$ is a multiplier of $\mathcal{D}_a$, and
  \begin{equation*}
    \|f_\mu\|_{\Mult(\mathcal{D}_a)} \approx \|f_\mu\|_\infty,
  \end{equation*}
  where the implied constants only depend on $a$ and $d$.
\end{prop}

\begin{proof}
  Since the multiplier norm dominates the supremum norm, we have to show the inequality ``$\lesssim$''.
  Let $f = f_\mu$ and
  \begin{equation}
    \label{eqn:f_r}
    f_r(z) = f(r z) = \int_{\partial \mathbb{B}_d} K_a(r z,w) \, d \mu(w)
    = \int_{\partial \mathbb{B}_d} K_a(z, rw) \, d \mu (w).
  \end{equation}
  We will show that $\|f_r\|_{\Mult(\mathcal{D}_a)} \lesssim \|f_r\|_{\infty}$ for all $0 < r <1$, where the implied constant
  is independent of $f$ and $r$.
  To simplify notation, write $k_w(z) = K_a(z,w)$.
  We will use Theorem \ref{thm:Sarason_function} and instead show that
  \begin{equation*}
    \sup_{z \in \mathbb{B}_d} \Re \langle f_r, k_z f_r \rangle \lesssim \|f_r\|_\infty^2.
  \end{equation*}

  Since the map
  \begin{equation*}
    \partial \mathbb{B}_d \to \mathcal{D}_a, \quad w \mapsto k_{r w},
  \end{equation*}
  is continuous, the integral on the right-hand side of \eqref{eqn:f_r} converges in $\mathcal{D}_a$.
  Thus, by the reproducing property of the kernel,
  \begin{align*}
    \langle  f_r, k_z f_r \rangle &= \int_{\partial \mathbb{B}_d}
    \langle k_{r w}, k_z f_r \rangle \, d \mu(w)
    = \int_{\partial \mathbb{B}_d} \overline{f_r(r w) k_z(rw)} \, d \mu(w).
  \end{align*}
  Therefore,
  \begin{align*}
    \Re \langle f_r, k_z f_r \rangle \le \|f_r\|_\infty \int_{\partial \mathbb{B}_d} |k_z(r w)| \, d \mu (w)
    &\lesssim \|f_r\|_\infty \Re \int_{\partial \mathbb{B}_d} k_z(r w) \, d \mu(w) \\
    &= \|f_r\|_\infty \Re f_r(z)
    \le \|f_r\|_\infty^2. \qedhere
  \end{align*}
\end{proof}

We are ready to provide the direct proof of Theorem \ref{thm:main_dirichlet}, which we restate
in equivalent form.

\begin{thm}
  Let $0 \le a < 1$ and let $E \subset \partial \mathbb{B}_d$ be compact.
  Then $E$ is $\Mult(\mathcal{D}_a)$-totally null if and only if $\capac{E}{\mathcal{D}_a} = 0$.
\end{thm}

\begin{proof}
    The ``only if'' part was already established in Proposition \ref{prop:only_if}.
  Conversely, suppose that $\capac{E}{\mathcal{D}_a} = 0$. By upper semi-continuity
  of capacity, there exists a decreasing sequence $(E_n)$ of compact neighborhoods of $E$
so that $ \lim_{n \to \infty} \capac{E_n}{\mathcal{D}_a} = 0$; see \cite[Theorem 2.1.6]{EKM+14}.
Let $\mu_n$ be a positive measure supported on $E_n$ as in Lemma \ref{lem:potential_basic}
and let $g^{(n)} = f_{\mu_n}$ be the corresponding holomorphic potential.
We claim that
\begin{enumerate}
  \item $\liminf_{r \nearrow 1} \Re g^{(n)}(r \zeta) \gtrsim 1$ for all $\zeta \in E$ and all $n \in \mathbb{N}$;
  \item the sequence $(g^{(n)})$ converges to $0$ in the weak-$*$ topology of $\Mult(\mathcal{D}_a)$.
\end{enumerate}

Indeed,
Part (1) is immediate from Lemma \ref{lem:potential_basic} (b).
To see (2), we first observe that Lemma \ref{lem:potential_basic} (c) and Proposition \ref{prop:multiplier_norm}
imply that the sequence $(g^{(n)})$ is bounded in multiplier norm.
Using Lemma \ref{lem:potential_basic} (a), we see that $\|g^{(n)}\|_{\cD_a}^2 \lesssim \capac{E_n}{\cD_a}$,
so $(g^{(n)})$ converges to zero in the norm of $\cD_a$ and in particular pointwise on $\bB_d$, hence (2) holds.

Let now $\nu$ be a positive $\Mult({\mathcal{D}_a})$-Henkin measure that is supported on $E$.
We will finish the proof by showing that $\nu(E) = 0$; see the discussion following Definition \ref{defn:Henkin_TN}.
Item (1) above and Fatou's lemma show that
\begin{equation*}
  \nu(E) = \int_{E} 1 \, d \nu \lesssim \liminf_{r \nearrow 1} \int_{\partial \mathbb{B}_d} \Re g^{(n)}(r \zeta) \, d \nu(\zeta).
\end{equation*}
Since $\nu$ is $\Mult(\mathcal{D}_a)$-Henkin, the associated
integration functional $\rho_\nu$ extends to a weak-$*$ continuous functional on $\Mult(\mathcal{D}_a)$,
which we continue to denote by $\rho_\nu$.
Since $\lim_{r \nearrow 1} g^{(n)}_r = g^{(n)}$ in the weak-$*$ topology of $\Mult(\mathcal{D}_a)$ by Lemma \ref{lem:dilations_convergence},
we find that for all $n \in \mathbb{N}$,
\begin{equation*}
  \lim_{r \nearrow 1} \int_{\partial \mathbb{B}_d} \Re g^{(n)}(r \zeta) d \nu(\zeta)
  = \Re \rho_\nu(g_n).
\end{equation*}
Thus,
\begin{equation*}
  \nu(E) \lesssim \Re \rho_\nu(g_n)
\end{equation*}
for all $n \in \mathbb{N}$. Taking the limit $n \to \infty$ and using Item (2),
we see that $\nu(E) = 0$, as desired.
\end{proof}

\section{Proof of Theorem \ref{thm:main_CNP}}
\label{sec:main_CNP}

In this section, we prove a refined version of Theorem \ref{thm:main_CNP}.
Let $\mathcal{H}$ be a regular unitarily invariant space on $\mathbb{B}_d$.
Recall that a compact set $E \subset \mathbb{B}_d$ is said to be an \emph{unboundedness set for $\mathcal{H}$} if
there exists $f \in \mathcal{H}$ with $\lim_{r \nearrow 1} |f(r \zeta)| = \infty$ for all $\zeta \in E$.
We also say that $E$ is a \emph{weak unboundedness for $\mathcal{H}$} if there exists a separable auxiliary Hilbert space $\mathcal{E}$
and $f \in \mathcal{H} \otimes \mathcal{E}$
so that $\lim_{r \nearrow 1} \|f(r \zeta)\| = \infty$ for all $\zeta \in E$.

\begin{thm}
  Let $\mathcal{H}$ be a regular unitarily invariant complete Pick space on $\mathbb{B}_d$.
  The following assertions are equivalent for a compact set $E \subset \mathbb{B}_d$.
  \begin{enumerate}[label=\normalfont{(\roman*)}]
    \item $E$ is $\Mult(\mathcal{H})$-totally null.
    \item $E$ is an unboundedness set for $\mathcal{H}$.
    \item $E$ is a weak unboundedness set for $\mathcal{H}$.
  \end{enumerate}
\end{thm}

\begin{proof}
  (i) $\Rightarrow$ (ii)
  Suppose that $E$ is totally null. In the first step, we will show that for each $M > 1$,
  there exists $f \in \mathcal{H} \cap A(\mathbb B_d)$ satisfying
  \begin{enumerate}
    \item $f \big|_E = M$,
    \item $\|f\|_\mathcal{H} \le 1$, and
    \item  $\Re f \ge 0$ on $\overline{\mathbb{B}_d}$.
  \end{enumerate}
  Let $\varepsilon = 1/M$.
  Since $E$ is totally null, the simultaneous Pick and peak interpolation result \cite[Theorem 1.5]{DH20}
  shows that there exists $\eta \in A(\mathcal{H}) \subset \mathcal H \cap A(\mathbb B_d)$ satisfying
  $\eta \big|_E = (1 - \varepsilon^2)^{1/2}$, $\eta(0) = 0$ and $\|\eta\|_{\Mult(\mathcal{H})} \le 1$.
  It follows that the column multiplier
  \begin{equation*}
    \begin{bmatrix}
      \varepsilon \\ (1 - \varepsilon^2)^{1/2} \eta
    \end{bmatrix}
  \end{equation*}
  has multiplier norm at most one, so the implication (b) $\Rightarrow$ (a) of part (i) of \cite[Theorem 1.1]{AHM+17c}
  implies that the function $f$ defined by
  \begin{equation*}
    f = \frac{\varepsilon}{1 - (1 - \varepsilon^2)^{1/2} \eta}
  \end{equation*}
  belongs to the closed unit ball of $\mathcal{H}$.
  Moreover, since $\|\eta\|_{\Mult(\mathcal{H})} \le 1$, we find that $|\eta(z)| \le 1$
  for all $\zeta \in \overline{\mathbb{B}_d}$,
  from which it follows that $f \in A(\mathbb B_d)$ and $\Re f \ge 0$.
  Clearly, $f \big|_E = \frac{1}{\varepsilon} = M$.
  This observation finishes the construction of $f$.

  The above construction yields, for each $n \ge 1$, a function $f_n \in \mathcal{H} \cap A(\mathbb B_d)$ satisfying
  $f_n \big|_E = 1$, $\|f_n\|_{\mathcal{H}} \le 2^{-n}$ and $\Re f_n \ge 0$.
  Define $f = \sum_{n=1}^\infty f_n \in \mathcal{H}$.
  Let $\zeta \in E$.
  Then for each $N \in \mathbb{N}$, we have that
  \begin{equation*}
    \liminf_{r \nearrow 1} \Re f(r \zeta) \ge \sum_{n=1}^N \Re f_n(\zeta) = N.
  \end{equation*}
  Thus, $\lim_{r \nearrow 1} |f(r \zeta)| = \infty$ for all $\zeta \in E$, so $E$ is an unboundedness set for $\mathcal{H}$.

  (ii) $\Rightarrow$ (iii) is trivial.

  (iii) $\Rightarrow$ (i)
  Suppose that $E$ is a weak unboundedness set for $\mathcal{H}$ and let $f \in \mathcal{H} \otimes \mathcal{E}$ satisfy $\|f\| \le 1$ and $\lim_{r \nearrow 1} \|f(r \zeta)\| = \infty$ for all $\zeta \in E$. 
  By the implication (a) $\Rightarrow$ (b) of part (i) of \cite[Theorem 1.1]{AHM+17c}, we may write $f = \frac{\Phi}{1 - \psi}$,
  where $\Phi \in \Mult(\mathcal{H}, \mathcal{H} \otimes \mathcal{E}),\psi \in \Mult(\mathcal{H})$ have multiplier norm at most $1$
  and $|\psi(z)| < 1$ for all $z \in \mathbb{B}_d$.
  In particular, $\|\Phi(z)\| \le 1$ for all $z \in \mathbb{B}_d$, hence $\lim_{r \nearrow 1} \psi(r \zeta) = 1$ for all $\zeta \in E$.

  Let now $\mu$ be a positive $\Mult(\mathcal{H})$-Henkin measure that is supported on $E$ and let $\rho_\mu$
  denote the associated weak-$*$ continuous integration functional on $\Mult(\mathcal{H})$.
  We have to show that $\mu(E) = 0$; see the discussion following Definition \ref{defn:Henkin_TN}.
  To this end,
  we write $\psi_r(z) = \psi(r z)$ and let $n \in \mathbb{N}$.
  Applying the dominated convergence theorem and the fact that
  $\lim_{r \nearrow 1} \psi_r^n = \psi^n$ in the weak-$*$ topology of $\Mult(\mathcal{H})$ (see Lemma \ref{lem:dilations_convergence}), we find that
  \begin{equation*}
    \mu(E) = \lim_{r \nearrow 1} \int_{E} \psi_r^n \, d \mu  = \lim_{r \nearrow 1} \int_{\partial \mathbb{B}_d} \psi_r^n \, d \mu = \rho_\mu(\psi^n).
  \end{equation*}
  Since $\psi$ is a contractive multiplier satisfying $|\psi(z)| < 1$ for all $z \in \mathbb{B}_d$, it follows that $\psi^n$ tends to zero
  in the weak-$*$ topology of $\Mult(\mathcal{H})$. So taking the limit $n \to \infty$ above, we conclude that $\mu(E) = 0$, as desired.
\end{proof}

Let us briefly compare the direct proof of the implication ``capacity $0$ implies totally null''
given in Section \ref{sec:Dirichlet} with the proof via Theorem \ref{thm:main_CNP}.
If $E \subset \partial \mathbb B_d$ is a compact set with $C_{s,2}(E) = 0$,
then the work of Ahern and Cohn \cite{AC89} and of Cohn and Verbitsky \cite{CV95}
shows that $E$ is unboundedness set for $\mathcal H_s$.
To show this, they use holomorphic potentials and their fundamental properties
(cf.\ Lemma \ref{lem:potential_basic}) to construct a function $f \in \mathcal H_s$ satisfying
$\lim_{r \nearrow 1} |f(r \zeta)| = \infty$ for all $\zeta \in E$.
Proceeding via Theorem \ref{thm:main_CNP}, one then applies the factorization
result \cite[Theorem 1.1]{AHM+17c} to $f$ to obtain a multiplier $\psi$ of $\mathcal H$ of norm
at most $1$ satisfying $\lim_{r \nearrow 1} \psi(r \zeta) = 1$ for all $\zeta \in E$,
from which the totally null property of $E$ can be deduced.

The direct proof given in Section \ref{sec:Dirichlet} uses holomorphic potentials as well, this time
to construct a sequence of functions in $\mathcal H$, which, roughly speaking,
have large radial limits on $E$ compared to their norm.
It is then shown that the holomorphic potentials themselves form a bounded sequence of multipliers,
from which the totally null property of $E$ can once again be deduced.

\bibliographystyle{amsplain}
\bibliography{literature}

\end{document}